\documentclass[11pt,leqno]{amsart}
\usepackage{amsmath,amssymb,amsthm}
\usepackage{hyperref}

\usepackage{color}

\usepackage{bbm}

\renewcommand{\geq}{\geqslant}
\renewcommand{\leq}{\leqslant}

\newcommand{\lip}{{\mathrm{lip}}_0}

\newcommand{\Lip}{{\mathrm{Lip}}_0}

\newcommand{\SA}{\operatorname{SNA}}

\newtheorem{theorem}{Theorem}[section]
\newtheorem*{maintheorem}{Main Theorem}
\newtheorem{lemma}[theorem]{Lemma}

\theoremstyle{definition}

\theoremstyle{remark}
\newtheorem{remark}[theorem]{Remark}
\newtheorem*{question}{Question}
\numberwithin{equation}{section}

\def\fnote#1{\footnote}

\def\ignora#1{}
\def\n3#1{\left\vert  \! \left\vert \! \left\vert \, #1 \, \right\vert \!
  \right\vert \! \right\vert }

\begin{document}

\title{ Infinite dimensional spaces in the set of strongly norm-attaining Lipschitz maps }

\author[Avil\'es]{Antonio Avil\'es}
\address[Avil\'es]{Universidad de Murcia, Departamento de Matem\'{a}ticas, Campus de Espinardo 30100 Murcia, Spain
	\newline
	\href{https://orcid.org/0000-0003-0291-3113}{ORCID: \texttt{0000-0003-0291-3113} } }
\email{\texttt{avileslo@um.es}}

\author[Mart\'inez-Cervantes]{Gonzalo Mart\'inez-Cervantes}
\address[Mart\'inez-Cervantes]{Universidad de Alicante, Departamento de Matem\'{a}ticas, Facultad de Ciencias, 03080 Alicante, Spain
	\newline
	\href{http://orcid.org/0000-0002-5927-5215}{ORCID: \texttt{0000-0002-5927-5215} } }	
\email{gonzalo.martinez@ua.es}

\author[Rueda Zoca]{Abraham Rueda Zoca}
\address[Rueda Zoca]{Universidad de Murcia, Departamento de Matem\'{a}ticas, Campus de Espinardo 30100 Murcia, Spain
	\newline
	\href{https://orcid.org/0000-0003-0718-1353}{ORCID: \texttt{0000-0003-0718-1353} }}
\email{\texttt{abraham.rueda@um.es}}
\urladdr{\url{https://arzenglish.wordpress.com}}

\author[Tradacete]{Pedro Tradacete}
\address[Tradacete]{Instituto de Ciencias Matem\'aticas (CSIC-UAM-UC3M-UCM)\\
	Consejo Superior de Investigaciones Cient\'ificas\\
	C/ Nicol\'as Cabrera, 13--15, Campus de Cantoblanco UAM\\
	28049 Madrid, Spain.
	\newline
	\href{http://orcid.org/0000-0001-7759-3068}{ORCID: \texttt{0000-0001-7759-3068} }}
\email{pedro.tradacete@icmat.es}



\maketitle

\markboth{AVIL\'ES, MART\'INEZ-CERVANTES, RUEDA ZOCA AND TRADACETE}{INFINITE DIMENSIONAL SPACES IN $\SA(M)$}

\begin{abstract}
We prove that if $M$ is an infinite complete metric space then the set of strongly norm-attaining Lipschitz functions $\SA(M)$ contains a linear subspace isomorphic to $c_0$. This solves an open question posed by V. Kadets and O. Rold\'an.
\end{abstract}

\section{Introduction}
Let $M$ be a metric space with a distinguished point $0 \in M$.
The couple $(M,0)$ is commonly called a \emph{pointed metric space}.
Given a pointed metric space $M$, we will denote by $\Lip(M)$ the Banach space of all real-valued Lipschitz functions on $M$ which vanish at $0$ under the standard Lipschitz norm
$$\Vert f\Vert:=\sup\left\{ \frac{\vert f(x)-f(y)\vert}{d(x,y)}\ :\ x,y\in M, x\neq y \right\} .$$
Observe that the requirement that every Lipschitz function in $\Lip(M)$ vanishes at a common point guarantees that the unique constant function that $\Lip(M)$ contains is the zero function and, consequently, the best Lipschitz constant of a function defines a norm on $\Lip(M)$. Moreover, since every Lipschitz function on a metric space $M$ extends uniquely to its completion, there is no loss of generality in assuming that $M$ is complete.

There is a natural concept of norm-attainment in $\Lip(M)$. Given $f\in \Lip(M)$, we say that $f$ \textit{strongly attains its norm} if the above supremum is actually a maximum, i.e. if there are two different points $x,y\in M$ such that $\frac{f(x)-f(y)}{d(x,y)}=\Vert f\Vert$. We denote by $\SA(M)$ the set of those Lipschitz functions which strongly attain their norm. Even though it seems the most natural sense of norm attainmnet, one reason for saying ``strong norm-attainment'' is that, thanks to the rich structure of the space $\Lip(M)$, many concepts of norm-attainment for a Lipschitz function have been considered in the literature (see \cite{ccm2020,god2015,kms2016} for other notions of norm-attainment). On the other hand, the reason for the terminology ``strong'' is that it is a very restrictive concept as we will explain below.

The origin of the strong norm-attainment goes back to the papers \cite{god2015,kms2016}. On the one hand, in \cite{kms2016} negative examples are provided to illustrate that this concept is very restrictive. To mention one example, in \cite[Theorem 2.3]{kms2016} a Lipschitz function $f\in \Lip([0,1])$ is given with the property that $d(f,\SA([0,1])\geq \frac{1}{2}$. On the other hand, in \cite[Section 5]{god2015} examples of metric spaces $M$ are provided for which the set $\SA(M)$ is dense in $\Lip(M)$. Since then, a vast literature has appeared dealing with the problem of when the set of vector-valued strongly norm-attaining Lipschitz functions (defined in the natural way) is dense in the space of all Lipschitz functions (see \cite{ccgmr2019,chiclana,cgmr2021,cm2019,cm2021,gppr2018,gpr2017,jmr2022}).

A different study of the size of $\SA(M)$ has been recently performed in \cite{karo} where, instead of studying when $\SA(M)$ is dense, the authors study when $\SA(M)$ contains linear subspaces. At first glance, one could expect that, since the strong-norm attainment is a restrictive concept and $\SA(M)$ can be small, the set $\SA(M)$ could fail to contain linear subspaces of dimension bigger than $2$ for some metric spaces $M$. However, in \cite{karo} it is proved that, for every infinite metric space $M$, the set $\SA(M)$ contains, for every $n\in\mathbb N$, a vector subspace isomorphic to $\ell_1^n$.

A natural question now (explicitly posed in \cite[Question 1]{karo}) is whether $\SA(M)$ contains an infinite-dimensional subspace when $M$ is infinite. From the results of \cite[Section 3]{karo} it is reasonable to conjecture that, if we look for an infinite-dimensional subspace of $\Lip(M)$ contained in $\SA(M)$, the natural candidate would be $c_0$. In fact, in \cite[Theorem 3]{karo} it is proved that if $M$ is $\sigma$-precompact, then all the subspaces that $\SA(M)$ may contain must be separable and isomorphic to polyhedral spaces. Observe also that, given any infinite metric space $M$, the space $\Lip(M)$ always contains an isomorphic copy of $c_0$ \cite{cdw2015}. Finally, observe that in \cite[Example 1]{karo} it is proved that $\SA(M)$ contains a linear isometric copy of $c_0$ if $M$ contains an isometric copy of $[0,1]$. For $M=[0,1]$ this says that $\SA(M)$ may fail to be dense in $\Lip(M)$ (and hence being somehow ``small''), but it contains an isometric copy of $c_0$, in particular it may contain infinite dimensional subspaces. All these results motivated the authors of \cite{karo} to pose the following question.

\begin{question} \cite[Question 2]{karo} Is it true that for every infinite complete pointed metric space $M$ the corresponding $\SA(M)$ contains an isomorphic copy of $c_0$?
\end{question}

The aim of this paper is to give a positive answer to this question. Consequently, the main theorem of the paper is the following.

\begin{maintheorem}\label{theorem:maintheorem}
Let $M$ be an infinite complete pointed metric  space. Then $\SA(M)$ contains an isomorphic copy of $c_0$.
\end{maintheorem}

After introducing some preliminary results and the necessary notation in Section \ref{section:notation}, we prove the Main Theorem in Section \ref{section:main}. For convenience let us sketch the strategy of the proof here: we begin by distinguishing whether the accumulation of the given metric space is infinite or finite (possibly empty); the former case will be considered in Theorem \ref{theo:inficluster}, while for the latter we consider first the situation when the metric space is countable and compact by means of Theorem \ref{theo:proper}, and when that is not the case, then up to removing a finite collection of balls we have to deal with a discrete metric space. In this situation, Theorem \ref{theo:discrenoudiscre} allows us to handle the case of non-uniformly discrete spaces, whereas the uniformly discrete case will be solved by distinguishing several possibilities on the structure of the underlying metric space and with the help of Ramsey's theorem.

\section{Notation and preliminary results}\label{section:notation}

All the metric spaces considered in this text will be assumed to be complete. Given a metric space $M$, we denote by $B(x,r)=\{y\in M: d(x,y)<r\}$ and $\overline{B}(x,r):=\{y\in M: d(x,y)\leq r\}$. Moreover, for any pair of subsets $A,B\subseteq M$ we write
$d(A,B)$ to denote the number
$$d(A,B)=\inf\{d(x,y):x\in A,~y\in B\}.$$

Observe that, if we select different distinguished points $0$ and $0'$ in a metric space $M$, then the mapping $f\longmapsto f-f(0')$ establishes an onto linear isometry from the space $\Lip(M)$ onto the space $\operatorname{Lip}_{0'}(M)$. Because of this, the choice of the origin is irrelevant and we will consider that all the metric spaces in the text are pointed without explicit mention.

Given a metric space $M$, we will say that $\SA(M)$ contains an isomorphic copy of $c_0$ (or similar sentences) if we can find a linear operator $T\colon c_0\longrightarrow \Lip(M)$ which is an isomorphism onto $T(c_0)$ and such  that $T(c_0)\subseteq \SA(M)$. The sequence $\{e_n\}$ always denotes the canonical basis of $c_0$.

In order to prove that a sequence of Lipschitz functions is equivalent to the $c_0$ basis the following lemma will be extremely useful. The proof follows directly from that of \cite[Lemma 1.5]{ccgmr2019}.

\begin{lemma}\label{lemma:disjsupportc0}
Let $M$ be a metric space and let $\{f_n\}$ be a sequence in the unit ball of $\Lip(M)$. Write, for every $n\in\mathbb N$, $U_n:=\{x\in M: f_n(x)\neq 0\}$. If $U_n\cap U_m=\emptyset$ for every $n\neq m$, then the operator $T:c_0\longrightarrow \Lip(M)$ given by $T(e_n)=f_n$ is bounded with $\Vert T\Vert\leq 2$.
\end{lemma}

Given a metric space $M$, we will denote by $M'$ the set of cluster points of $M$. The space $M$ is said to be \textit{discrete} if $M'=\emptyset$; in other words, if for every $x\in M$ there exists $r_x>0$ such that $B(x,r_x)=\{x\}$. If a common radius $r$ can be found, i.e. if there exists $r>0$ such that $d(x,y)\geq r$ whenever $x\neq y$, then $M$ is said to be \textit{uniformly discrete}.
Given $\varepsilon>0$, we say that a sequence $\{x_n\}$ in $M$ is $\varepsilon$-separated if $d(x_n,x_m)>\varepsilon$ for every $n\neq m$. We will make use of the well-known fact that every sequence in a metric space contains either a Cauchy subsequence or an $\varepsilon$-separated subequence for some $\varepsilon>0$. As a consquence every infinite metric space contains an infinite sequence of pairwise disjoint balls.

The following lemma is proved in \cite[Lemma 4.3]{ccgmr2019} and will be the key ingredient when dealing with discrete metric spaces which are not uniformly discrete.

\begin{lemma}\label{lemma:noudiscretejfa}
Let $M$ be a complete metric space. Assume that $M$ is discrete but not uniformly discrete. Then, for every $k\geq 2$ and every $\varepsilon>0$, there exist $x, y\in M$ such that $0<d(x,y)\leq\varepsilon$ and the set $M\setminus \overline{B}(x,k\,d(x,y))$ is not uniformly discrete.
\end{lemma}



\section{Main results}\label{section:main}

We start with a useful lemma which will cover several situations when $\SA(M)$ contains an isometric copy of $c_0$.

\begin{lemma}\label{lemma:technicalinficluster}
Let $M$ be a metric space. Assume that there is a sequence $B(x_n,R_n)$ $n\in\mathbb N$ of balls satisfying the following conditions:
\begin{enumerate}
    \item $d(B(x_i,R_i),B(x_j,R_j))>0$ for every $n\neq m$,
    \item $\frac{R_i+R_j}{d(B(x_i,R_i),B(x_j,R_j))}<\frac{1}{2}$ for every $i\neq j$, and
    \item for every $n\in\mathbb N$ there is $y_n\in B(x_n,R_n)\setminus\{x_n\}$ such that $d(x_n,y_n)<d(y_n, M\setminus B(x_n,R_n))$.
\end{enumerate}
Then, for every $n\in\mathbb N$, there is a norm-one Lipschitz function $f_n$ with $\frac{f_n(y_n)-f_n(x_n)}{d(x_n,y_n)}=1$, so that $\{f_n\}$ is isometric to the $c_0$-basis and $\SA(M)$ contains $\overline{\operatorname{span}}\{f_n\}$.
\end{lemma}

\begin{proof}

By removing one ball, we can assume that $0\notin B(x_n,R_n)$ for every $n\in\mathbb N$. Let us consider for $n\in\mathbb N$ the function $$f_n(x):=\min\{d(x,x_n),d(x,M\backslash B(x_n,R_n))\}.$$ Due to the conditions on $B(x_n,R_n)$, it is straightforward to check that $\Vert f_n\Vert\leq1$ and that $f(y_n)-f(x_n)=d(y_n,x_n)$. Hence, $f_n$ strongly attains the norm at the pair $x_n, y_n$.

Let us now prove the remaining part. To this end, pick $(\lambda_n)$ a sequence in $c_0$ and write the formal sum $g:=\sum_{n=1}^\infty \lambda_n f_n$. Observe that $g:M\longrightarrow \mathbb R$ is well defined because the supports of ${f_n}$ are contained in $B(x_n,R_n)$ and these are pairwise disjoint. Let us prove that $g$ is Lipschitz, its norm is $\max_{i\in\mathbb N} \vert \lambda_i\vert$, and that $g$ strongly attains its Lipschitz norm. This will conclude the rest of the proof.

In order to do so, pick $i\in\mathbb N$ such that $\vert\lambda_i\vert=\max_{j\in\mathbb N}\vert \lambda_j\vert$. Pick $x\neq y\in M$, and let us estimate $A:=\left\vert \frac{g(x)-g(y)}{d(x,y)}\right\vert$. There are different possibilities for the position of $x$ and $y$.
\begin{enumerate}
    \item If $x,y\notin \bigcup_{n\in\mathbb N} B(x_n,R_n)$ it is immediate that $g(x)=g(y)=0$ and then $A=0$ in this case.
    \item If $x\in \bigcup_{n\in\mathbb N} B(x_n,R_n)$ but $y\notin \bigcup_{n\in\mathbb N} B(x_n,R_n)$ we conclude that $g(y)=0$. Moreover, if $x\in B(x_n, R_n)$ for certain $n$ we obtain
    $$A=\vert \lambda_n\vert \frac{\vert f_n(x)\vert }{d(x,y)}=\vert\lambda_n\vert \frac{\vert f_n(x)-f_n(y)\vert }{d(x,y)}\leq \vert\lambda_n\vert.$$
    \item If $x\in B(x_n,R_n), y\in B(x_m,R_m)$ for certain $n,m\in \mathbb N$, then $g(x)=\lambda_n f_n(x)$ and $g(y)=\lambda_m f_m(y)$. If $n=m$ then it is immediate that $A\leq\vert\lambda_n\vert$. If $n\neq m$, observe that $g(x)=\lambda_n f_n(x)=\lambda_n(f_n(x)-f_n(x_n))$ and, similarly, $g(y)=\lambda_m f_m(y)=\lambda_m(f_m(y)-f_m(x_m))$ by the definition of $f_n$ and $f_m$. Then
    \[
    \begin{split}
        A &\leq \frac{\vert \lambda_n\vert \vert f_n(x_n)-f_n(x)\vert+\vert\lambda_m\vert \vert f_m(x_m)-f_m(y)\vert}{d(x,y)} \\& \leq  \frac{\vert \lambda_n\vert d(x_n,x)+\vert\lambda_m\vert d(x_m,y)}{d(x,y)}\\
        & \leq |\lambda_i| \frac{R_n+R_m}{d(B(x_n,R_n),B(x_m,R_m))}\\
        & < \frac{1}{2}\vert \lambda_i\vert.
    \end{split}
    \]
\end{enumerate}
Observe that this covers all the possible possitions for $x$ and $y$, which proves that $g$ is Lipschitz and $\Vert g\Vert\leq \vert \lambda_i\vert=\max_{n\in\mathbb N}\vert \lambda_ n\vert$. Observe that
$$\frac{\vert g(x_i)-g(y_i)\vert}{d(x_i,y_i)}=\vert\lambda_i\vert.$$
This proves simultaneously that $\|g\|=\max_{n\in\mathbb N}\vert \lambda_ n\vert$, and that it attains its norm at $x_i,y_i$, which finishes the proof.
\end{proof}

Let us obtain a number of theorems which are consequence of the above lemma. Let us start with the following one.

\begin{theorem}\label{theo:inficluster}
Let $M$ be a complete metric space so that $M'$ is infinite. Then $\SA(M)$ contains an isometric copy of $c_0$.
\end{theorem}

\begin{proof}
Since $M'$ is infinite, there exists a sequence of distinct points $\{x_n\}$ in $M'$ and numbers $r_n>0$ such that $\{B(x_n,r_n)\}$ is an infinite family of pairwise disjoint balls. Taking $R_n=\frac{r_n}{7}$ and any $y_n \in B(x_n,\frac{R_n}{3})\setminus \{x_n\}$, standard computations show that the sequence of balls $\{B(x_n,R_n)\}$ fulfils conditions (1), (2) and (3) in Lemma \ref{lemma:technicalinficluster}, so we are done.
\end{proof}

Recall that a metric space $M$ is said to be \textit{proper} (or boundedly compact) if every closed and bounded set is compact.

Let us now consider two subspaces of $\Lip(M)$ of capital importance in the paper \cite{dal2}.  Given a metric space $M$ define
$$\lip(M):=\left\{f\in \Lip(M): \lim\limits_{\varepsilon\rightarrow 0} \sup\limits_{0<d(x,y)<\varepsilon} \frac{\vert f(x)-f(y)\vert}{d(x,y)}=0\right\},$$
$$S(M):=\left\{ f\in \lip(M) : \lim\limits_{r\rightarrow \infty} \sup_{{x \text{ or } y\notin B(0,r)}} \frac{\vert f(x)-f(y)\vert}{d(x,y)}=0\right\}.$$
The space $\lip(M)$ is known as the \textit{little Lipschitz function spaces \cite[Chapter 4]{weaver}} whereas, up to our knowledge, the introduction of the space $S(M)$ is originally from \cite{dal2}. It is proved in \cite{dal1} (respectively, in \cite{dal2}) that if $M$ is countable and compact (respectively, countable and proper), then $\lip(M)^{**}=\Lip(M)$ (respectively $S(M)^{**}=\Lip(M)$). We will take this into account to prove the following result.

\begin{theorem}\label{theo:proper}
Let $M$ be an infinite proper metric space. Then $\SA(M)$ contains an isomorphic copy of $c_0$.
\end{theorem}

\begin{proof}
If $M'$ is infinite then Theorem \ref{theo:inficluster} applies and $\SA(M)$ even contains isometric copies of $c_0$. If $M'$ is finite, it is immediate that $M$ is countable and proper. Then $S(M)^{**}=\Lip(M)$ and, in particular, $S(M)$ is infinite-dimensional. It is clear by a compactness argument (and it is proved in \cite[Lemma 2.4]{dal2}) that $S(M)\subseteq \SA(M)$. In order to finish the proof, observe that \cite[Lemma 3.9]{dal2} implies that $S(M)$ is $(1+\varepsilon)$-isometric to a subspace of $c_0$, from where $S(M)$ contains an isomorphic copy of $c_0$.
\end{proof}

Another case where Lemma \ref{lemma:technicalinficluster} applies is the case when $M$ is discrete but not uniformly discrete thanks to Lemma \ref{lemma:noudiscretejfa}. Consequently, we get the following theorem.
\begin{theorem}\label{theo:discrenoudiscre}
Let $M$ be an infinite discrete metric space which is not uniformly discrete. Then $\SA(M)$ contains an isometric copy of $c_0$.
\end{theorem}

\begin{proof}
Since $M$ is not uniformly discrete, there exist a sequence of pair of points $\{x_n,y_n\}$ with $\{d(x_n,y_n)\}$ being a strictly decreasing sequence converging to zero. The sequence $\{x_n\}$ has a convergent subsequence or either a subsequence which is $\varepsilon$-separated for some $\varepsilon>0$.
Notice that if $\{x_{n_k}\}$ is convergent to a point $x$, then $\{y_{n_k}\}$ is also convergent to $x$ and $\{x_{n_k}\}$ or $\{y_{n_k}\}$ is not eventually constant, which yields a contradiction since $M$ is discrete.
Thus, passing to a subsequence if necessary, we may suppose that there exists $\varepsilon>0$ such that $d(x_n,x_k)>\varepsilon$ for every $n\neq k$ and $d(x_n,y_n)<\frac{\varepsilon}{14}$ for every $n$.

Set $R=\frac{\varepsilon}{7}$. We claim that the sequence of balls $\{B(x_n, R)\}$ satisfies the conditions of Lemma \ref{lemma:technicalinficluster}.
To this end, pick $n\neq k$ in $\mathbb N$.
First, notice that since the sequence $\{x_n\}$ is $\varepsilon$-separated, we have $$d(B(x_n,R),B(x_k,R))\geq d(x_n,x_k)-2R\geq \frac{5}{7}\varepsilon>\frac{4}{7}\varepsilon=4R>0,$$ which shows that $\{B(x_n, R)\}$ satisfies conditions (1) and (2) in Lemma \ref{lemma:technicalinficluster}
Furthermore, it is plain that $d(x_n,y_n)<\frac{\varepsilon}{14}=\frac{R}{2}\leq d(y_n,M\setminus B(x_n,R))$, which completes the proof.
\end{proof}

Let us end with the last preliminary lemma in order to prove the main result. To this end, we need the following definition. Given $x\in M$ define
$$R(x):=\sup\{R\geq 0: \overline{B}(x,R)=\{x\}\}.$$

\begin{lemma}\label{lemma:sucec0attainR(x)}
Let $M$ be a metric space. Assume that there exists a pair of sequences $\{x_n\}, \{y_n\}$ of points of $M$ such that $x_n\neq y_n$ for every $n\in \mathbb N$ and with the following properties:
\begin{enumerate}
    \item $\{x_n,y_n\}\cap \{x_m,y_m\}=\emptyset$ if $n\neq m$,
    \item $d(x_n,y_n)<R(x_n)+R(y_n)$ holds for every $n \in\mathbb N$.
\end{enumerate}
Then $\SA(M)$ contains an isomorphic copy of $c_0$.
\end{lemma}

\begin{proof} We can assume, up to removing one pair, that $0\notin \{x_n,y_n\}$ for every $n\in\mathbb N$. Now define $f_n:M\longrightarrow \mathbb R$ as follows:
$$f_n(x):=\left\{\begin{array}{cc}
\frac{R(x_n)}{R(x_n)+R(y_n)}d(x_n,y_n)   & \mbox{ if } x=x_n, \\
-\frac{R(y_n)}{R(x_n)+R(y_n)}d(x_n,y_n)     & \mbox{ if }x=y_n,\\
0 & \mbox{otherwise.}
\end{array} \right.$$
Observe that $f_n$ is Lipschitz for every $n\in\mathbb N$ because $Rx_n,\,Ry_n>0$, thus $x_n$ and $y_n$ are isolated points of $M$.

Let us prove that $\{f_n\}$ is a bounded sequence in $\Lip(M)$ which is equivalent to the $c_0$ basis and such that, for every $(\lambda_n)\in c_0$, the function $g=\sum_{n=1}^\infty \lambda_n f_n$ attains its Lipschitz norm. To this end, take $(\lambda_n)$ and $g$ as above, and assume the non-trivial case that $(\lambda_n)$ is not the zero sequence.

Since $(\lambda_n)\rightarrow 0$ and the sequence $\frac{d(x_n,y_n)}{R(x_n)+R(y_n)}$ is clearly bounded we conclude that $(\lambda_n \frac{d(x_n,y_n)}{R(x_n)+R(y_n)})\rightarrow 0$. Call $K:=\max\{\vert \lambda_n\vert \frac{d(x_n,y_n)}{R(x_n)+R(y_n)}: n\in\mathbb N\}$, and observe that $K\leq \Vert (\vert\lambda_n\vert)\Vert_\infty$. By the convergence to $0$ condition there exists a finite set $B\subseteq \mathbb N$ and $\varepsilon_1>0$ such that, if $n\in B$, $\vert \lambda_n\vert \frac{d(x_n,y_n)}{R(x_n)+R(y_n)}=K$, whereas $\vert \lambda_n\vert \frac{d(x_n,y_n)}{R(x_n)+R(y_n)}\leq (1-\varepsilon_1)K$ if $n\notin B$. On the other hand, for every $n\in\mathbb N$ we have that $\frac{d(x_n,y_n)}{R(x_n)+R(y_n)}<1$ by the assumption. Consequently, since $B$ is finite there exists $\varepsilon_2>0$ such that $\frac{d(x_n,y_n)}{R(x_n)+R(y_n)}\leq (1-\varepsilon_2)$ holds for every $n\in B$.
Observe that if $n\in B$ then $\lambda_n$ cannot be $0$.

Set $\varepsilon_0:=\min\{\varepsilon_1,\varepsilon_2\}<1$. Since $(\lambda_n)\rightarrow 0$ there exists $m\in\mathbb N$ such that $k\geq m$ implies $\vert \lambda_k\vert<\varepsilon_0 K\leq \varepsilon_0\Vert (\lambda_n)\Vert_\infty$.

Now let $A:=\{x_1,y_1,\ldots, x_m,y_m\}\subseteq M$. Let us prove that $$\sup_{u\neq v}\frac{\vert g(u)-g(v)\vert}{d(u,v)}=\sup_{(u,v)\in A^2, u\neq v}\frac{\vert g(u)-g(v)\vert}{d(u,v)}.$$ Since $A^2$ is finite, this will be enough to ensure that the previous supremum is actually a maximum, which would yield that $g\in \SA(M)$.

In order to do so, take $u\neq v$ and assume that $(u,v)\notin A^2$. We can assume, up to a relabeling, that $v\notin A$. Observe that $g(u)-g(v)=0$ unless either $u$ or $v$ belongs to $\{x_n,y_n: n\in\mathbb N\}$. Moreover,
$$\frac{\vert g(u)-g(v)\vert}{d(u,v)}\leq \frac{\vert g(u)\vert+\vert g(v)\vert}{d(u,v)}.$$
Let us obtain an upper bound for the latter.

Observe that, if for certain $n\in\mathbb N$ we get $u\in \{x_n,y_n\}$ then $\vert g(u)\vert=\vert \lambda_n\vert \frac{R(u)}{R(x_n)+R(y_n)}d(x_n,y_n)$. Similarly, if $v\in \{x_k,y_k\}$ for $k>m$, we get that $\vert g(v)\vert=\vert\lambda_k\vert \frac{R(v)}{R(x_k)+R(y_k)}d(x_k,y_k)$. Moreover, notice that $d(u,v)\geq \max\{R(u),R(v)\}$. Hence, it is clear that
\[
\begin{split}
    \frac{\vert g(u)\vert+\vert g(v)\vert}{d(u,v)}& \leq \sup\limits_{n\in\mathbb N} \vert \lambda_n\vert \frac{d(x_n,y_n)}{R(x_n)+R(y_n)}+\sup\limits_{k\geq m+1}\vert \lambda_k\vert \frac{d(x_k,y_k)}{R(x_k)+R(y_k)}
\end{split}\]
Let $k\geq m+1$ and $n\in\mathbb N$. On the one hand, if $n\in B$ then
$$\vert\lambda_n\vert \frac{d(x_n,y_n)}{R(x_n)+R(y_n)}+\vert \lambda_k\vert \frac{d(x_k,y_k)}{R(x_k)+R(y_k)}\leq \vert \lambda_n\vert(1-\varepsilon_0)+\vert\lambda_k\vert\leq \Vert (\vert \lambda_n\vert)\Vert_\infty.$$
On the other hand, if $n\notin B$ we get that
$$\vert\lambda_n\vert \frac{d(x_n,y_n)}{R(x_n)+R(y_n)}+\vert \lambda_k\vert \frac{d(x_k,y_k)}{R(x_k)+R(y_k)}\leq (1-\varepsilon_0)K+\vert \lambda_k\vert\leq K\leq \Vert (\vert\lambda_n\vert)\Vert_\infty.$$
Consequently, we get that $\sup\limits_{(u,v)\notin A^2, u\neq v}\frac{\vert g(u)-g(v)\vert}{d(u,v)}\leq \Vert (\lambda_n)\Vert_\infty$. Now observe that, given $i\in\mathbb N$ such that $\vert\lambda_i\vert=\Vert (\vert \lambda_n\vert)\Vert_\infty$ (notice that necessarily $i<m$), we have that $\frac{g(x_i)-g(y_i)}{d(x_i,y_i)}=\vert \lambda_i\vert=\Vert (\vert\lambda_n\vert)\Vert_\infty$, which shows that $\|g\|$ is Lipschitz and strongly attains its norm at some pair of points $(u,v) \in A^2$.

Furthermore, it is easy to check that $\|f_n\|=1$. Consequently, given $(\lambda_n)\in c_0$, we have by the above estimates and by Lemma \ref{lemma:disjsupportc0} that
$$\Vert (\lambda_n)\Vert_\infty\leq \left\Vert \sum_{n=1}^\infty \lambda_n f_n\right\Vert\leq 2 \Vert (\lambda_n)\Vert_\infty,$$
which proves that the operator $T:c_0\longrightarrow \Lip(M)$ given by $T(e_n)=f_n$ is an into isomorphism with $T(c_0)\subseteq \SA(M)$ as desired.
\end{proof}

Now we are ready to prove the Main Theorem.

\begin{proof}[Proof of Main Theorem]

The proof will be completed by distinguishing cases. First of all, if $M'$ is infinite then the conclusion follows from Theorem \ref{theo:inficluster}.

Consequently, we can assume that $M'$ is finite $M'=\{a_1,\ldots, a_p\}$ (with the convention $p=0$ if $M'=\emptyset$). It is clear that $M\setminus \bigcup\limits_{i=1}^p B(a_i,r)$ is discrete for every $r>0$. If, for every $r>0$ the previous set is finite, then $M$ is countable and compact, and then Theorem \ref{theo:proper} yields that $\SA(M)$ contains $c_0$ isomorphically.

So we can assume that, for some $r>0$, the set $N:=M\setminus \bigcup\limits_{i=1}^p B(a_i,r)$ is infinite.

Moreover, if there were $r>0$ such that $M\setminus \cup_{i=1}^p B(a_i,r)$ were discrete but not uniformly discrete, then the proof of Theorem \ref{theo:discrenoudiscre} would allow us to construct a sequence of disjoint balls in the assumptions of Lemma \ref{lemma:technicalinficluster} and we would conclude the result. Hence we can assume that $M\setminus \cup_{i=1}^p B(a_i,r)$ is uniformly discrete for every $r>0$.

We will finish the proof of this case by a systematic application of Lemma \ref{lemma:sucec0attainR(x)} with an extra discussion of the different possibilities for the structure of $M$. Observe that, if $x\notin M'$, then $R(x)>0$ and, in particular, the set of those elements with $R>0$ are all but finitely many. Observe also that $R(x)\leq d(x,y)$ for every $y\neq x$.

Let us also observe that if $R(z_n)\rightarrow 0$ then $z_n$ has a subsequence converging to $a_i$ for some $i$. Indeed, if $d(z_n,a_i)\geq r_0$ for some $r_0$ and for every $i$, since $M\setminus \bigcup\limits_{i=1}^p B(a_i, r_0)$ is uniformly discrete then $R(z_n)$ would be bounded from below by a positive constant.

Now we distinguish several cases:

\bigskip

\underline{Case 1}: There is a sequence $\{x_n\}\subseteq M$ of distinct points satisfying that there exists $y_n\in M$ such that $d(x_n,y_n)=R(x_n)$.

\medskip

Observe that in this case, up to removing finitely many terms of the sequence, we can assume $x_n\neq 0$ and $R(x_n)>0$ for every $n\in\mathbb N$.

Consider the following sets:
$$W:=\{\{n,m\}\in [\mathbb N]^2: \{x_n,y_n\}\cap \{x_m,y_m\}=\emptyset\},$$
$$G:=\{\{n,m\}\in [\mathbb N]^2: y_n=y_m\},$$
$$O:=\{\{n,m\}\in [\mathbb N]^2: x_n=y_m\mbox{ or }x_m=y_n\},$$
where for a set $S$ we write $[S]^2$ to denote the family of subsets of $S$ of cardinality $2$.
By Ramsey's Theorem there is an infinite subset $S$ of $\mathbb N$ such that $[S]^2$ is contained in one of the previous sets. Observe that we can discard that this set is $O$ because the only possibility would be that infinitely many $y_n's$ coincide, and then we reduce ourselves to the case that $G$ is the set found by Ramsey's Theorem. Now we have two possibilities:

\medskip

\underline{Case 1, subcase 1:} The conclusion of Ramsey's Theorem applies to $W$.\vspace{0.3cm}

Passing to a subsequence, we can assume that $\{x_n,y_n\}\cap \{x_m,y_m\}=\emptyset$ if $n\neq m$. Moreover, we can assume, up to removing finitely many $y_n$'s, that $R(y_n)>0$ for every $n\in\mathbb N$ and $y_n\neq 0$. In this case, we can apply Lemma \ref{lemma:sucec0attainR(x)} to the sequences $\{x_n\} ,\{y_n\}$ because
$$\frac{d(x_n,y_n)}{R(x_n)+R(y_n)}=\frac{R(x_n)}{R(x_n)+R(y_n)}<1$$
for every $n\in\mathbb N$ since both $R(x_n)$ and $R(y_n)$ are positive.

\medskip

\underline{Case 1, subcase 2}: The conclusion of Ramsey's Theorem applies to $G$.\vspace{0.3cm}

In this case we can assume, up to passing to a subsequence, that $y_n=y_m=y$ holds for every $n\neq m$.

Now a new application of Ramsey's Theorem is needed. To this end, consider
$$V:=\{\{n,m\}\in [\mathbb N]^2: d(x_n,x_m)=d(x_n,y)+d(x_m,y)\}$$
and set
$$Z:=\{\{n,m\}\in [\mathbb N]^2: d(x_n,x_m)<d(x_n,y)+d(x_m,y)\}.$$
Ramsey's Theorem implies that either $V$ or $Z$ contains the set $[S]^2$ for certain infinite set $S\subseteq \mathbb N$. If $[S]^2\subseteq Z$ then we can assume that, up to taking a further subsequence, $d(x_n,x_k)<d(x_n,y)+d(x_k,y)=R(x_n)+R(x_k)$ holds for every $n,k\in \mathbb N, n\neq k$. In this case, we can apply Lemma \ref{lemma:sucec0attainR(x)} to the sequences $\{x_{2n}\}$ and $\{x_{2n+1}\}$.

Thus, to finish the proof of this subcase we only need to show that if $[S]^2\subseteq V$ then $\SA(M)$ contains an isomorphic copy of $c_0$.
Passing to a subsequence if necessary, we can assume that $d(x_n,x_m)=d(x_n,y)+d(x_m,y)$. We will follow a different approach. In this case, given $n\in\mathbb N$, define $f_n:M\longrightarrow \mathbb R$ by the equation $f_n(x_n)=d(x_n,y)$ and $f_n(x)=0$ if $x\neq x_n$. Clearly, $\{f_n\}$ is $1$-equivalent to the $c_0$ basis.
Namely, let $(\lambda_n)\in c_0$ and set $g=\sum_{n=1}^\infty \lambda_n f_n$. Fix $i\in\mathbb N$ such that $\vert \lambda_i\vert:=\max_{n\in\mathbb N}\vert \lambda_n\vert$. Then $$\frac{\vert g(x_i)-g(y)\vert}{d(x_i,y)}=|\lambda_i|,$$
so the Lipschitz norm of $g$ is at least $|\lambda_i|$.
Given $u \in M$, notice that $g(u)\neq0$ if and only if $u \in \{x_n: n\in \mathbb{N}\}$. Thus, if $g(v)=0$ we have that $$\frac{\vert g(x_n)-g(v)\vert}{d(x_n,v)}= |\lambda_n|\frac{d(x_n,y)}{d(x_n,v)} \leq |\lambda_n|.$$
Since
$$\frac{\vert g(x_n)-g(x_m)\vert}{d(x_n,x_m)}\leq |\lambda_i|\frac{d(x_n,y)+d(x_m,y)}{d(x_n,x_m)} = |\lambda_i|,$$
we obtain that $\|g\|=|\lambda_i|$ and that $g\in \SA(M)$ as desired, which concludes the proof of case 1.

\bigskip

\underline{Case 2}: There is a sequence of distinct points $\{x_n\}\subseteq M$ such that, for every $n\in\mathbb N$, $R(x_n)<d(x_n,y)$ for every $y\neq x_n$.
\vspace{0.3cm}

In this case we claim that for every $n\in\mathbb N$ we can find infinitely many $y\in M$ such that $d(x_n,y)< R(x_n)+R(y)$. Fix $n\in\mathbb N$ and assume, by contradiction, that $\{y\in M: d(x_n,y)<R(x_n)+R(y)\}$ were finite. Notice that for every $k\in\mathbb N$, we can find $y_k$ such that $d(x_n,y_k)\leq R(x_n)+\frac{1}{k}$. Observe that we can assume that $y_k\neq y_j$ if $k\neq j$ since the value $R(x_n)$ is never attained, so we can suppose, up to removing finitely many elements, that $R(x_n)+R(y_k)\leq d(x_n,y_k)$ holds for every $k\in\mathbb N$. Thus, $R(x_n)+R(y_k)\leq R(x_n)+\frac{1}{k}$ for every $k$, which implies that $R(y_k)\rightarrow 0$. Passing to a subsequence we can assume that $\{y_k\}$ converges to $a_i$ for some $i$, from where $R(x_n)=\lim_k d(x_n,y_k)= d(x_n,a_i)$, which yields a contradiction since $R(x_n)$ is not attained.

The above claim allows us to construct by an inductive argument a sequence $y_n$ such that the inequality $d(x_n,y_n)<R(x_n)+R(y_n)$ holds for every $n\in\mathbb N$ and so that $\{x_n,y_n\}\cap \{x_m,y_m\}=\emptyset$ if $n\neq m$. Moreover, we can assume that $R(y_n)>0$ up to removing finitely many terms. Now, an application of Lemma \ref{lemma:sucec0attainR(x)} to the sequences $\{x_n\}$ and $\{y_n\}$ finishes the proof of the theorem.
\end{proof}

\begin{remark}
Observe that, as a consequence of the Main Theorem, we reprove the known result that if $M$ is infinite then $\Lip(M)$ contains an isomorphic copy of $c_0$, proved in \cite[Theorem 3.2]{cdw2015}. This result was later improved in \cite[Theorem 5]{CJ2017}, where an isometric version is proved (indeed, it is shown that $\Lip(M)$ contains isometric copies of $\ell_\infty$). We do not know whether an isometric version of the Main Theorem holds, i.e. whether $\SA(M)$ contains an isometric copy of $c_0$ whenever $M$ is an infinite metric (complete) space.
\end{remark}

\section*{Acknowledgements}

Research partially supported by Fundaci\'on S\'eneca - ACyT Regi\'on de Murcia [20797/PI/18]. The first three authors were also supported by MTM2017-86182-P (funded by MCIN/AEI/10.13039/501100011033 and ``ERDF A way of making Europe'').

The research of A. Rueda Zoca was also supported by MCIN/AEI/10.13039/\\ 501100011033 (Spain) Grant PGC2018-093794-B-I00, by Junta de Andaluc\'ia Grant A-FQM-484-UGR18 and by Junta de Andaluc\'ia Grant FQM-0185.

P.~Tradacete was also partially supported by Agencia Estatal de Investigaci\'on (AEI) and Fondo Europeo de Desarrollo Regional (FEDER) through grants PID2020-116398GB-I00 and CEX2019-000904-S funded by MCIN/AEI/ 10.13039/501100011033.

\end{document}